\documentclass[reqno]{amsart}
\usepackage{times,amssymb,amsmath,amsfonts,float,nicefrac,subfigure,float,accents,color,stmaryrd,float}
\usepackage{hyperref}
\usepackage{color}

\newcommand{\remove}[1]{}

\newcommand{\bitem}{\begin{itemize}}
\newcommand{\eitem}{\end{itemize}}
\newcommand{\benum}{\begin{enumerate}}
\newcommand{\eenum}{\end{enumerate}}
\newcommand{\beq}{\begin{equation}}
\newcommand{\eeq}{\end{equation}}

\newtheorem{theorem}{Theorem}[section]
\newtheorem{proposition}[theorem]{Proposition}

\theoremstyle{definition}
\newtheorem{definition}[theorem]{Definition}

\newcommand{\ip}[2]{\langle#1,#2\rangle}

\theoremstyle{remark}
\newtheorem{remark}{Remark}[section]

\numberwithin{equation}{section}

\newcommand\nc\newcommand
\nc\bfa{{\boldsymbol a}}\nc\bfA{{\bf A}}\nc\cA{{\mathcal A}}
\nc\bfb{{\boldsymbol b}}\nc\bfB{{\bf B}}\nc\cB{{\mathcal B}}
\nc\bfc{{\boldsymbol c}}\nc\bfC{{\bf C}}\nc\cC{{\mathcal C}}\nc\sC{{\mathscr C}}
\nc\bfd{{\boldsymbol d}}\nc\bfD{{\bf D}}\nc\cD{{\mathcal D}}
\nc\bfe{{\boldsymbol e}}\nc\bfE{{\bf E}}\nc\cE{{\mathcal E}}
\nc\bff{{\boldsymbol f}}\nc\bfF{{\bf F}}\nc\cF{{\mathcal F}}
\nc\bfg{{\boldsymbol g}}\nc\bfG{{\bf G}}\nc\cG{{\mathcal G}}
\nc\bfh{{\boldsymbol h}}\nc\bfH{{\bf H}}\nc\cH{{\mathcal H}}
\nc\bfi{{\boldsymbol i}}\nc\bfI{{\bf I}}\nc\cI{{\mathcal I}}\nc\sI{{\mathscr I}}
\nc\bfj{{\boldsymbol j}}\nc\bfJ{{\bf J}}\nc\cJ{{\mathcal J}}
\nc\bfk{{\boldsymbol k}}\nc\bfK{{\bf K}}\nc\cK{{\mathcal K}}
\nc\bfl{{\boldsymbol l}}\nc\bfL{{\bf L}}\nc\cL{{\mathcal L}}
\nc\bfm{{\boldsymbol m}}\nc\bfM{{\bf M}}\nc\cM{{\mathcal M}}
\nc\bfn{{\boldsymbol n}}\nc\bfN{{\bf N}}\nc\cN{{\mathcal N}}
\nc\bfo{{\boldsymbol o}}\nc\bfO{{\bf O}}\nc\cO{{\mathcal O}}
\nc\bfp{{\boldsymbol p}}\nc\bfP{{\bf P}}\nc\cP{{\mathcal P}}\nc\eP{{\EuScript P}}
\nc\bfq{{\boldsymbol q}}\nc\bfQ{{\bf Q}}\nc\cQ{{\mathcal Q}}
\nc\bfr{{\boldsymbol r}}\nc\bfR{{\bf R}}\nc\cR{{\mathcal R}}
\nc\bfs{{\boldsymbol s}}\nc\bfS{{\bf S}}\nc\cS{{\mathcal S}}
\nc\bft{{\boldsymbol t}}\nc\bfT{{\bf T}}\nc\cT{{\mathcal T}}
\nc\bfu{{\boldsymbol u}}\nc\bfU{{\bf U}}\nc\cU{{\mathcal U}}
\nc\bfv{{\boldsymbol v}}\nc\bfV{{\bf V}}\nc\cV{{\mathcal V}}
\nc\bfw{{\boldsymbol w}}\nc\bfW{{\bf W}}\nc\cW{{\mathcal W}}
\nc\bfx{{\boldsymbol x}}\nc\bfX{{\bf X}}\nc\cX{{\mathcal X}}
\nc\bfy{{\boldsymbol y}}\nc\bfY{{\bf Y}}\nc\cY{{\mathcal Y}}
\nc\bfz{{\boldsymbol z}}\nc\bfZ{{\bf Z}}\nc\cZ{{\mathcal Z}}
\newcommand\reals{{\mathbb R}}

\begin{document}

\author[A. Barg]{Alexander Barg$^\MakeLowercase{a}$}\thanks{$^a$
Department of ECE and Institute for Systems Research, University
of Maryland, College Park, MD 20742, and IITP, Russian Academy of
Sciences, Moscow, Russia. Email: abarg@umd.edu. Research supported
in part by NSF grants CCF1217894, DMS1101697, and NSA
98230-12-1-0260;}
\author[A. Glazyrin]{Alexey Glazyrin$^\MakeLowercase{b}$}\thanks{$^b$ Department of Mathematics, The University
of Texas, Brownsville, TX 78520, Email: Alexey.Glazyrin@utb.edu. Research supported in part by NSF grants DMS1101688, DMS1400876}
\author[K. A. Okoudjou]{Kasso A.~Okoudjou$^\MakeLowercase{c}$}\thanks{$^c$ Department of Mathematics, University
of Maryland, College Park, MD 20742, Email: kasso@math.umd.edu.  Research supported in part by a RASA from the Graduate School of UMCP, and by a grant from the Simons Foundation ($\# 319197$ to Kasso Okoudjou).}
\author[W.-H. Yu]{Wei-Hsuan Yu$^\MakeLowercase{d}$}\thanks{$^{d}$ Department of Mathematics, University
of Maryland and Inst. for Systems Research, College Park, MD
20742, Email: mathyu@math.umd.edu.  Research supported in part by
NSF grants CCF1217894,  DMS1101697.}

\title[Finite two-distance tight frames]{Finite two-distance tight frames}
\date{\today}

\begin{abstract}
A finite collection of unit vectors $S\subset\reals^n$ is called a
spherical two-distance set if there are two numbers $a$ and $b$
such that the inner products of distinct vectors from $S$ are
either $a$ or $b$. We prove that if $a\ne -b,$ then a two-distance set that forms
a tight frame for $\reals^n$ is a spherical embedding of a 
strongly regular graph. We also describe all two-distance tight frames obtained from a given graph.
Together with an
earlier work by S. Waldron on the equiangular case ({\em Linear Alg. Appl.}, vol. 41, pp. 2228-2242, 2009)
this completely characterizes two-distance tight frames. 
As an intermediate result, we obtain a classification of all two-distance 2-designs.\end{abstract}

\maketitle

\section{Introduction}

A finite collection of unit vectors $S\subset\reals^n$ is
called a spherical two-distance set if there are two numbers $a$
and $b$ such that the inner products of distinct vectors from $S$
are either $a$ or $b$. If in addition $a=-b,$ then $S$ defines a
set of equiangular lines through the origin in $\reals^n$.  Equiangular
lines form a classical subject in discrete geometry following
foundational papers of Van Lint, Seidel, and Lemmens
\cite{lint66,lem73}. Equiangular line sets are closely related to strongly
regular graphs and two-graphs \cite{del77b,Godsil2001} which form the main source
of their constructions. Another group of results is concerned
with bounding the maximum size $g(n)$ of spherical
two-distance sets in $n$ dimensions.
    We refer to \cite{barg13,barg14} for the latest results on upper bounds on $g(n)$ as well as an overview of the
    relevant literature.

A finite collection of vectors $S=\{x_1,\dots,x_N\}\subset \reals^n$ is
called a \textit{finite frame} for the Euclidean space $\reals^n$ if
there are constants $ 0 < A \leq B < \infty$ such that for all $x
\in \reals^n$
  \begin{equation}\label{eq:frame-bounds}
A||x||^2 \leq \sum_{i=1}^N \langle x,x_i \rangle^2 \leq B
||x||^2.
\end{equation}
If $A=B,$ then $S$ is called an $A$-tight frame, in which case
  \begin{equation}\label{eq:A}
    A=\frac1n {\sum_{i} \|x_i\|^2}.
  \end{equation}
  It is trivially  seen that a finite collection of vectors $S=\{x_i: i=1,\dots N\}\subset \reals^n$ is an $A$-tight frame if and only if for any $x
\in \reals^n$, 

\begin{equation}\label{tight-frame-char}
A x = \sum_{i=1}^N \langle x,x_i \rangle x_i.
\end{equation}
  
 If in addition
$\|x_i\|=1$ for all $i \in I$, then $S$  is a finite unit-norm tight
frame or FUNTF. If at the same time $S$ is a spherical
two-distance set, we call it a {\em two-distance tight frame}. In
particular, if the two inner products in $S$ satisfy the condition
$a=-b,$ then it is an {\em equiangular tight frame} or ETF. All frames in this paper
will be assumed unit-norm.

The Gram matrix $G$ of $S$ is defined by $G_{ij}=\langle
x_i,x_j\rangle, 1\le i,j\le N,$ where $N=|S|.$ If $S$ is a FUNTF
for $\reals^n$, then it is straightforward to show  \cite{hklw} that $G$ has one
nonzero eigenvalue $\lambda=N/n$ of multiplicity $n$ and 
eigenvalue 0 of multiplicity $N-n$.

Frames have been used in signal processing and have a large number
of applications in sampling theory, wavelet theory, data
transmission, and filter banks \cite{ck12, koche1}.  The
study of ETFs was initiated by Strohmer and Heath \cite{stro03}
and Holmes and Paulsen \cite{hol04}. In particular, \cite{hol04}
shows that equiangular tight frames give error correcting codes
that are robust against two erasures. Bodmann et al. \cite{bod07}
show that ETFs are useful for signal reconstruction when all the
phase information is lost. Sustik et al. \cite{sus07} derived
necessary conditions on the existence of ETFs as well as bounds on
their maximum cardinality.

Benedetto and Fickus \cite{ben03} introduced a useful parameter of
the frame, called the {\em frame potential}. For our purposes it
suffices to define it as $FP(S)=\sum_{i,j=1}^{N}\langle
x_i,x_j\rangle^2.$ For a two-distance frame we obtain
  \begin{equation}\label{eq:fp1}
  \sum_{i,j=1}^N \langle x_i,x_j\rangle^2=N+2\nu_a a^2+(N(N-1)-2\nu_a)b^2,
  \end{equation}
where $\nu_a=|\{(i, j):  i<j: \ip{x_i}{x_j}=a\}|.$ 
\remove{Moreover, if
$N>2n+1,$ Theorem \ref{thm:LRS} implies that $b=(ka-1)/(k-1),$
where $k$ is an integer between $2$ and $(1/2)(1+\sqrt{2n}).$ This
gives some information for a lower bound on $FP(S)$, but
fortunately, a more general and concrete result is known from
\cite{ben03}.}
\begin{theorem}{\cite[Theorem.6.2]{ben03}}\label{thm:BF}
   Let $N>n.$  If $S=\{x_i: 1\leq i \leq N\}$ is any set of unit-norm vectors, then
  \begin{equation}\label{eq:fp}
    FP(S)\ge \frac {N^2} n
  \end{equation}
  with equality if and only if $S$ is a tight frame.
\end{theorem}

A finite collection of unit vectors $S=\{x_i: i=1,\dots,N\}$ in $\reals^n$ is called a 
{\em spherical $2$-design} \cite{del77b} if
   \begin{equation}
   \label{eq:design}
   \sum_{i=1}^N x_i=0, \quad \sum_{i,j=1}^N\langle x_i,x_j\rangle^2=\frac{N^2}n.
   \end{equation} 
In other words, a spherical 2-design is a FUNTF with the center of mass at the origin.   

\begin{remark} In \cite{ban14} spherical sets that satisfy only the tight frame condition (the second condition in
\eqref{eq:design}) are called spherical designs of harmonic index 2. In the sequel we will refer to such spherical designs as shifted $2$-designs. 
\end{remark}

To state our main result we need several definitions. A regular graph of degree $k$ on $v$ vertices is called {\em strongly
regular}  if every two adjacent vertices have $\lambda$ common
neighbors and every two non-adjacent vertices have $\mu$ common
neighbors. Below we use the notation $\text{SRG}(v,k,\lambda,
\mu)$ to denote such a graph. Note that the complement of a strongly regular graph $\text{SRG}(v,k,\lambda,
\mu)$,  is also strongly regular, namely $\text{SRG}(v,v-k-1, v-2k+\mu-2, v-2k+\lambda)$. The theory
of strongly regular graphs is presented, for instance, in \cite{Godsil2001,Brouwer12}. 
Below we use a classical construction of spherical embeddings of strongly regular graphs introduced by Delsarte, Goethals, and Seidel \cite[Example 9.1]{del77b}; see also \cite{neu81,ban05}. Roughly speaking,
a spherical embedding of $\Gamma=\text{SRG}(v,k,\lambda,\mu)$ is obtained by associating a basis of $\reals^v$ with the vertices
of $\Gamma$ and projecting these vectors on an eigenspace of the adjacency matrix of $\Gamma$. A more detailed 
description is 
given in Sect.~\ref{sect:proof} after we develop all the necessary pieces of notation.

   In this paper we characterize two-distance FUNTFs by linking them to spherical 2-designs
and strongly regular graphs. Our main result is as follows:
\begin{theorem}\label{thm:main} Let $S=\{x_i: i\in I\}$ be a non-equiangular 
two-distance FUNTF in $\reals^n.$ Then $S$ forms a spherical two-distance
$2$-design or a shifted $2$-design. In either case $S$ can be obtained as a spherical embedding of 
a strongly regular graph. Under spherical embedding, every strongly regular graph gives rise to three different two-distance $2$-designs and therefore, to six different 
two-distance tight frames, two of which are regular simplices.
\end{theorem}
The proof is given in Sections \ref{sect:first},~\ref{sect:proof}. As an intermediate result (see Theorem \ref{thm:2-design}), we fully characterize spherical 2-designs that form spherical two-distance sets.

Strongly regular graphs form examples of classical objects in algebraic combinatorics called association schemes \cite{Brouwer12}. 
Although we do not use the language of schemes in this paper, we note that our results contribute to the study of the general problem
of characterizing spherical designs that can be obtained from association schemes.

Note that the connection between equiangular line sets and strongly regular graphs
is well known (Seidel et al.~\cite{Seidel73,del77b};
see also \cite{Godsil2001}). It has been recently addressed in the
context of frame theory, particularly in the study of ETFs
\cite{stro03,hol04}. A recent paper by Waldron \cite{wal09} proves that an ETF in $\reals^n$
with $N\ge n$ vectors exists if and only if there exists an $\text{SRG}(N-1,k,(3k-N)/2,k/2),$ 
where $k$ is a certain function of $n$ and $N$.  Furthermore, \cite{wal09} also contains  many examples of ETFs in $\reals^n, n\le 50$.
Together with Theorem \ref{thm:main} this result completes the description of two-distance tight frames,
equiangular or not.

\section{Basic properties}\label{sect:first}

We begin with an  easy example of 2-distance FUNTFs  which is given by the following construction.
We will need the following theorem.
\begin{theorem}[Larman, Rogers, and Seidel, \cite{lar77}] \label{thm:LRS}
Let $S$ be a spherical two-distance set in $\mathbb{R}^n$. If $|S|
> 2n+1$ then the inner products $a,b$ are related by the equation
$b=(ka-1)/(k-1)$ where $k\in\{2,\dots, \lfloor(1+\sqrt{2n})/2\rfloor\}$ is an integer.
\end{theorem}
The original proof of \cite{lar77} had $2n+3$ in place of $2n+1,$
while the above improvement is due to Neumaier \cite{neu81}. Given the value of $a,$
we denote by $b_k(a)$ the corresponding value of $b$.  


  \begin{proposition}\label{prop:simplex} Let $e_1,\dots,e_{n+1}$ be the standard basis in
$\reals^{n+1}$. The projection of the set
   \begin{equation}\label{eq:simp}
     S=\{e_i+e_j, 1\le i<j\le n+1\}
  \end{equation}
on the hyperplane $x_1+\dots+x_{n+1}=2$ forms a two-distance tight frame for $\reals^n$.
\end{proposition}
\begin{proof}
Note that the inner products of distinct vectors in $S$ are either 1 or 0.
Let
  \begin{gather*}
    \nu_{1,1}=|\{(i,j): i<j,\,\langle e_1+e_2,e_i+e_j\rangle=1\}|
   \end{gather*}
Observe that $(i,j)$ is contained in this set if and only if $i=1$
or $i=2,$ and we obtain $\nu_{1,1}=2(n-1).$ By symmetry, the value
$\nu_{1,1}$ does not depend on the choice of the fixed vector
$e_1+e_2,$ so the total number of unordered pairs of vectors in
$S$ with inner product 1 equals
  $$
   \nu_1=\frac 12 \binom {n+1}2 \nu_{1,1}=\frac 12(n-1)n(n+1).
  $$
The pairs of distinct vectors not counted in $\nu_1$ are orthogonal,
and their number is
  $$
   \nu_0=\binom{n(n+1)/2}{2}-\nu_1=\frac18 (n-2)(n-1)n(n+1).
 $$
Now let us project the vectors of $S$ on the plane
$x_1+\dots+x_{n+1}=2$ and scale the result to place them on the
unit sphere around the point $z_0=\frac2{n+1}(1,1,\dots,1).$ Since each vector $z_{i,j}=e_i+e_j$  already belongs to the plane, the resulting vector will be $z'_{i,j}=c(z_{i,j}-z_0)$ where $c= \tfrac{1}{\sqrt{2 - \frac{4}{n+1}}}.$ By a series of computations we see that $$\ip{z'_{i,j}}{z'_{k,l}}=c^2(\ip{z_{i,j}}{z_{k,l}} -\tfrac{4}{n+1}).$$ And since $\ip{z_{i,j}}{z_{k,l}}  \in \{0,1\}$ we arrive that the fact that $\ip{z'_{i,j}}{z'_{k,l}} \in \{a, b\}$ with $a=(n-3)/(2(n-1))$ and  $b=-2/(n-1)$.


This
information suffices to compute the frame potential, and we obtain
  $$
    FP(S)= N+2 \nu_1a^2+2\nu_0 b^2=\frac
    {N^2}{n}
    $$
The frame potential meets the lower bound \eqref{eq:fp} with
equality, which implies that $S$ forms a FUNTF for $\reals^n.$
\end{proof}

We next give   a characterization
result for two-distance FUNTFs.
\begin{definition}  Let $S\subset \reals^{n}$ be a spherical
two-distance set with inner products $a$ and $b$, $b<a,$ let $x_i\in S,$ and let
   $$
   N_{a,i}=|\{j: x_j\in S, \langle x_i,x_j\rangle=a\}|.
   $$
$S$ is called {\em regular} if $N_{a,i}$ does not depend on $i$. For regular sets
we denote this quantity simply by $N_a.$
\end{definition}
%
\begin{theorem}\label{thm:na} Let $S\subset \reals^n, |S|=N$ be a two-distance FUNTF 
with inner products $a$ and $b$ such that $a^2-b^2\ne 0.$ 
Then $S$ is regular and 
  \begin{gather}\label{eq:Na}
  N_a=\frac{({N}/{n})-1 -(N-1)b^2}{a^2-b^2}\\
-n(a+b)-nab(N-1)=N-n \quad\text{or}\quad (N-n)(a+b)-nab(N-1)=N-n.  
  \label{eq:b}
  \end{gather}
\end{theorem}
\begin{proof} $G$ is similar to a diagonal matrix of order $N$ with $n$ nonzero entries $\lambda=N/n$ on the diagonal.
 Therefore, $G^2- \lambda G =0,$ so $G^2 = \lambda G$ and $(G^2)_{ii}=\lambda$ for all $i$ since $G_{ii}=1.$
We also have $(G^2)_{ii}=\sum_{j=1}^N G_{ij}^2,$ so the norm of
every row and of every column is the same and equals
$\sqrt\lambda.$

Now let $N_a$ be the number of entries $a$ in any fixed column.
Then
 $$
    1+a^2 N_a+b^2(N-1-N_a)=\frac N n.
 $$
 This implies \eqref{eq:Na}.
 
Thus, ${\bf 1}=(11\dots1)$ is an
eigenvector of the Gram matrix $G$ with eigenvalue 0 or $N/n$.
Suppose it is the former, then $G\cdot{\bf1}=0,$ so the sum of
entries in every row is 0. This implies that
   $1+aN_a+(N-1-N_a)b=0.$
   By substituting $N_a$ given by~\eqref{eq:Na} into this last equation, and after some simplifications  we obtain the first of the two options for $b$ in~\eqref{eq:b}.

Now suppose that $G \cdot \mathbf{1} = \frac{N}{n} \mathbf{1},$ so
the sum of entries of $G$ in any given row equals $N/n.$ Repeating
the calculation performed for the first case, we obtain the second
of the two possibilities for $b.$
\end{proof}

\begin{remark} Another way to express the alternative in \eqref{eq:b} is as follows. 
The sum of squared entries of every row of $G$ equals $N/n$ and the sum of the entries is either
$0$ or $N/n$. These two equations translate into the two conditions for $a$ and $b.$
   
   In the next section we characterize FUNTF for each of the two
cases in \eqref{eq:b}.
\end{remark}
\begin{remark} If $a=-b,$ then the statement of Theorem \ref{thm:na} does not hold.
Indeed, consider the set $S=\{x_1,\dots,x_{28}\}$ of 28 vectors in
$\mathbb{R}^7$ constructed as in \eqref{eq:simp}. By
Theorem \ref{thm:LRS} the inner products between distinct vectors
in $S$ are $\pm1/3,$ so they form a set of equiangular lines. For
any given vector $x\in S$ we have $|\{y\in S: \langle
x,y\rangle=1/3\}|=12$ and $|\{y\in S: \langle
x,y\rangle=-1/3\}|=15.$ Now consider the set $S'=\{-x_1, x_2,
\dots,x_{28}\}$ which is also a FUNTF with inner products
$\pm1/3,$ but the first column of $G$ contains 12 entries equal to
$-1/3,$ which is different from all the other columns.
\end{remark}

\section{Two-distance FUNTFs and strongly regular graphs} \label{sect:proof}
Connections between equiangular line sets and ETFs on the one side
and strongly regular graphs on the other are well known and have
been used in the literature to characterize the sets of parameters
of ETFs \cite[Ch.~11]{Godsil2001}, \cite{wal09}. In this section
we extend this connection by relating two-distance (non
equiangular) FUNTFs, 2 designs, and strongly regular graphs.
%


We begin with a necessary condition for the existence of
two-distance FUNTFs. Let $S$ be such a frame. The Gram matrix of
any two-distance set with inner products $a,b$ can be written as
    \begin{equation}\label{eq:F}
    G=I+a\Phi_1+b\Phi_2,
    \end{equation}
where $\Phi_1$ and $\Phi_2$ are the corresponding indicator matrices. 
We also denote by $\Gamma_1$ and $\Gamma_2$ the graphs with adjacency matrices $\Phi_1$ and $\Phi_2,$ respectively. 



\begin{proposition}\label{prop:tf-design}
If $S$ is a $2$-distance FUNTF in $\reals^n$ with inner products $a,b$, then $S$ is either an $n$-dimensional spherical $2$-design, or is similar to an $(n-1)$-dimensional spherical $2$-design contained in a subsphere of radius $\sqrt{1-1/n}$. In the former 
(resp., latter) case $a$ and $b$ satisfy the first (resp., second) equality in \eqref{eq:b}.
\end{proposition}

\begin{proof} Let $S=\{x_i: 1\le i\le N\}$ and let $s=\sum_{i=1}^N x_i.$
 Then for each $i$, $1\leq i\leq N$ the value $t:=\langle x_i, s\rangle$
 does not depend on $i$ and is equal to $t=N_a a + (N-N_a) b + 1,$
where $N_a$ is given in \eqref{eq:Na}. 
 
Applying~\eqref{tight-frame-char} for $x=s,$ we obtain 
   $$
   \frac N n s = \sum_{i=1}^N t x_i = t s.
   $$
Hence either $s=0$ and $S$ is a spherical $2$-design, or $t=\frac N n$
and then $\langle s, s \rangle = N t = \frac {N^2} n.$

Suppose that $s\ne 0$ (equivalently $t=N/n$). For each $i$, $1\leq i\leq N$, denote $y_i=\frac {x_i-s/N}{\sqrt{1-1/n}}$.
We will show that the set $S'=\{y_i: i=1,\dots,N\},$ which is similar to the set $S,$ forms a spherical $2$-design in $\reals^{n-1}.$ This will imply that $S$ lies on a sphere of radius $\sqrt{1-1/n}$ in $\reals^{n}.$

First we check that $\langle y_i,s\rangle = 0$ for all $i$. Indeed,
   \begin{equation}\label{eq:0}
\langle y_i, s \rangle = \frac {\langle x_i, s \rangle - {\langle s, s \rangle}/ N}{\sqrt{1-1/n}}= \frac {N/n - \frac{N^2/n}{N}}{\sqrt{1-1/n}}=0.
  \end{equation}
This establishes that $S'$ is an $(n-1)$-dimensional set.
Moreover, $S'$ lies on the unit sphere. Indeed, using that $\langle y_i, s\rangle = 0,$ we obtain
  $$
  \|y_i\|^2
   =\frac {\|x_i\|^2-\|s/N\|^2} {1-1/n} = \frac {1-\frac {N^2/n}{N^2}} {1-1/n}=1.$$
Clearly $S'$ is a two-distance set. It remains to show that $S'$ forms a 2-design \eqref{eq:design}.
The center-of-masses condition is clearly satisfied. 
To check the tight frame condition let us compute the frame potential of $S'$ and use Theorem \ref{thm:BF}.
We have
   \begin{align*}
   \frac {N^2}{n} &= \sum\limits_{i,j=1}^N |\langle x_i, x_j \rangle|^2 = 
    \sum\limits_{i,j=1}^N \Big|\Big\langle \sqrt{1-
   \frac1n}\ y_i +\frac sN,  \sqrt{1-\frac1n}\ y_i +\frac sN \Big\rangle\Big|^2\\
   &= \sum\limits_{i,j=1}^N \Big(\Big(1-\frac1n\Big)\langle y_i, y_j\rangle + \frac{\|s\|^{2}}{N^{2}}\Big)^2\\
   &= \Big(1-\frac1n\Big)^2 FP(S') +\frac 2n \sum\limits_{i,j=1}^N \langle y_i, y_j\rangle + \frac{N^2}{n^2}\\
   &=\Big(1-\frac1n\Big)^2 FP(S')+ \frac{N^2}{n^2}
   \end{align*}
where the last step uses the condition $\sum_i y_i=0$.
Thus, $FP(S')=\frac {N^2} {n-1}$ and therefore, $S'$ is an $(n-1)$-dimensional $2$-design.

Finally, note that $t$ is an eigenvalue of $G$, namely, $G\cdot{\bf1}=t\bf1.$ Recalling that the two cases in \eqref{eq:b}
correspond to $t=0$ and $t=N/n,$ we obtain the final claim of the proposition.
\end{proof}

Observe that a related result was proved in \cite{noz12}. Namely, Theorem 4.7 in that paper states (in our terms) that a spherical
set $S\subset \reals^n$ is a 2-design if and only if $G\cdot{\bf1}=0$ and $G^2=\frac{\sum_{x\in S}\|x\|^2}n G.$

\vspace*{.1in}Due to the Delsarte-Goethals-Seidel theorem (\cite[Theorem 7.4]{del77b}), any spherical two-distance $2$-design is associated with a 
strongly regular graph and therefore, due to Proposition \ref{prop:tf-design}, any two-distance tight frame, too, is associated with
a strongly regular graph. To keep our exposition self-contained we give a short direct proof of this fact.
\begin{proposition}\label{prop:2-tight-srg}
If $S$ is a two-distance tight frame with inner products $a$ and $b$, $a^2-b^2\ne 0$, then its associated graph $\Gamma_1$ (and $\Gamma_2$ as the complement of $\Gamma_1$) is a strongly regular graph.
\end{proposition}

\begin{proof}
It follows from~\eqref{tight-frame-char} and~\eqref{eq:A} that for any two vectors $x_k$, $x_l$ of $S$, 
   \begin{equation}\label{eq:Nn}
   \frac N n \langle x_k, x_l\rangle = \sum_{i=1}^N \langle x_k, x_i \rangle \langle x_i, x_l \rangle.
   \end{equation}
Fix indices $k$ and $l$ and assume $\langle x_k, x_l\rangle = a$. Let 
  $$
  I_{\alpha,\beta}=\{i\in \{1,\dots,N\}: \langle x_k, x_i \rangle = \alpha\text{ and }\langle x_i, x_l \rangle = \beta\},
    $$
    where $\alpha,\beta\in \{a,b\},$ and let $C_a:=|I_{a,a}|.$  Note that by the symmetry of the Gram matrix $G$, we have that $ |I_{a,b}|=|I_{b,a}|$.
Let us find the cardinality of $I_{a,b},$ i.e., the set of indices $i$ with the entry $a$
in row $k$ and entry $b$ in row $l.$ Consider the subset of indices $i$ in row $k$
of $G$ with $\ip{x_k}{x_i}=a$ except $i=l$ (in this position row $l$ contains 1). There are $N_a-1$
such indices, where $N_a$ is the number of $a$'s in the row (see the remark before Theorem \ref{thm:na}).
 We then need to subtract the number of those $i$ for which $\ip{x_i}{x_l}=a$. But those are precisely the indices in the set $I_{a,a}.$ Consequently, 
    \begin{align}\label{eq:ab}
  |I_{a,b}|=|I_{b,a}|=N_a -C_a-1. 
  \end{align}
We next observe that the union of the (disjoint)
sets $I_{\alpha, \beta}$ $\alpha, \beta \in \{a, b\}$ gives all the indices in row $k$ of $G$ 
except the diagonal entry. Therefore, we obtain 
  \begin{equation}\label{eq:aabb}
  N-1=|I_{a,a}|+2|I_{a,b}|+|I_{b,b}|.
  \end{equation}
Recall that we have $N_b+N_a=N-1.$ Taking this together with \eqref{eq:ab} and \eqref{eq:aabb} and performing
simplifications, we obtain
  \begin{align*}
 |I_{b,b}|=N_b-N_a+C_a+1.
  \end{align*}
We can then rewrite~\eqref{eq:Nn} as 
  \begin{align*}
    \frac N n a &= 2(N_a-C_a-1)ab+C_a a^2 + (N_b-N_a+C_a+1)b^2 \\
   &= 2(N_a-1) ab + (N_b-N_a+1)b^2 + C_a(a-b)^2.
  \end{align*}
Since $a\neq b$, there is a unique $C_a$ that satisfies this equality. In other words, any pair of connected vertices of the associated graph $\Gamma_1$ has the same number $C_a$ of common neighbors. Similarly, any two non-connected vertices of $\Gamma_1$ have the same number $C_b$ of common neighbors. Therefore, $\Gamma_1$ is a strongly regular graph.
\end{proof}
We now set out to describe all two-distance tight frames. Propositions \ref{prop:tf-design} and \ref{prop:2-tight-srg} 
imply that we just need to find all spherical two-distance embeddings of strongly regular graphs and check the
$2$-design conditions for them.

\vspace*{.1in}\noindent {\em Spherical embeddings of SRGs.} Let $\Gamma_1$ be an $\text{SRG}(v,k,\lambda,\mu)$ which is not a complete or empty graph  and let $\Phi_1$ be its
adjacency matrix. As already mentioned, the set of vertices of $\Gamma_1$
can be embedded in the sphere by projecting the vectors of the standard basis of the
space $\reals^v$ on the eigenspaces
of $\Phi_1.$ 

The spectral structure of the matrix $\Phi_1$ is as follows.
It has three mutually orthogonal eigenspaces that correspond to three eigenvalues:
the all-one vector $\mathbf{1}$ 
with eigenvalue $k$, an eigenspace ${E_1}$ of dimension $n_1$ with eigenvalue $r_1$, and an eigenspace 
${E_2}$ of dimension $n_2$ with eigenvalue $r_2$ \cite[p.117]{Brouwer12}. Note that for Tur{\'a}n graphs and
their complements it is possible that $r_1=k.$
The values of $n_1, r_1, n_2, r_2$ can be found explicitly via the parameters $(v,k,\lambda,\mu).$ Since these values
are useful in constructing examples, we quote the expressions for them from \cite[pp.219-220]{Godsil2001}:
  \begin{align*}
  r_{1,2}&=\frac12(\lambda-\mu\pm\sqrt{(\lambda-\mu)^2+4(k-\mu)}\,)\\
  n_{1,2}&=\frac12\Big(v-1\mp\frac{2k+(v-1)(\lambda-\mu)}{\sqrt{(\lambda-\mu)^2+4(k-\mu)}}\Big).
  \end{align*}

Geometrically the Delsarte-Goethals-Seidel construction amounts to projecting orthogonally the standard basis vectors of $\reals^v\! $ on an eigenspace, for instance ${E_1},$
 and normalizing the projections (they all have the same length) to obtain unit lengths. 
Denote the obtained spherical set by  $S_1=S_1(\Gamma_1)$ and denote its two inner products by $a_1$ and $b_1,$ so that
$\Gamma_1$ is the graph of inner products $a_1, b_1;$ cf. \eqref{eq:F}.
  It is easy to show \cite{del77b} that this spherical set supports an $n_1$-dimensional $2$-design.

 Similarly we can obtain an $n_2$-dimensional $2$-design 
 $S_2=S_2(\Gamma_1)$ with inner products $a_2$ and $b_2$ by projecting $\reals^N$ 
 on ${E_2}$ and normalizing the projections. Finally, let $S_0$ denote the trivial one-dimensional embedding and
 note that $a_0=b_0=1$.

Let $\Gamma_2$ be the complement graph of $\Gamma_1$ and let $\Phi_2$ be its adjacency matrix.
We have
  $$
  \Phi_2=J-I-\Phi_1.
  $$
The vector $\bf1$ is an eigenvector of each of these matrices, and any vector $z$ such that $\langle z,{\bf1}\rangle=\sum z_i=0$ is
an eigenvector of $J$ and $I$. Hence if such vector $z$ is an eigenvector of $\Phi_1$, it is also an eigenvector of $\Phi_2.$
Thus, the matrices $\Phi_1$ and $\Phi_2$ share the same spectral structure. In particular, $\Phi_2$  also has three eigenvalues and three eigenspaces that coincide with the eigenspaces of $\Phi_1$: a vector of all ones $\mathbf{1}$ 
with eigenvalue $v-1-k$, an eigenspace ${E_1}$ of dimension $n_1$ with eigenvalue $s_1$, an eigenspace 
${E_2}$ of dimension $n_2$ with eigenvalue $s_2$.

\begin{proposition}\label{prop:srg-embed} Let $\Gamma_1(N,k,\lambda,\mu)$ be a strongly regular graph that is not complete or empty.
For any two-distance spherical embedding $S=\{x_1, \ldots,x_N\}$ of $\Gamma_1,$ 
there are three nonnegative real numbers $\alpha, \beta, \gamma$, $\alpha^2+\beta^2+\gamma^2=1$
such that for all $i=1,\dots,N$
   \begin{equation}\label{eq:a}
   x_i=\alpha x_i(0) + \beta x_i(1) + \gamma x_i(2)
   \end{equation}
for some $x_i(j), j=0,1,2.$ The sets $S_j(\Gamma_1)=\{x_i(j): 1\le i\le N\}, j=0,1,2$ form the Delsarte-Goethals-Seidel spherical embeddings of $\Gamma_1$ and are
contained in mutually orthogonal unit spheres of dimensions 
$1$, $n_1,$ and $n_2,$ respectively.
\end{proposition}

\begin{proof}
Let $S$ be a two-distance spherical embedding of $\Gamma_1$ with 
distances $a$ and $b$. Write the Gram matrix $G$ of $S$ as in \eqref{eq:F}.
The embedding $S$ exists if and only if $G$ is positive semidefinite. 
Since the matrices $\Phi_1$ and $\Phi_2$ share the spectral structure, 
we can find all eigenvalues of $G$ and check their non-negativity. This results in the following inequalities:
\begin{equation}\label{eq:sp}  \begin{aligned}
1+a k + b(N-1-k) &\geq 0\\
1+a r_1 + b s_1 &\geq 0\\
1+ a r_2 + b s_2 &\geq 0,
   \end{aligned}
   \end{equation}
(some of these inequalities may trivialize to $1\ge 0$). 
The set of all feasible pairs $(a,b)$ is the intersection of at most three half-planes in the plane. Note that this set
must belong to the square $[-1,1]^2,$ so it is bounded. Moreover, $G\succeq 0$ if and only if the inequalities \eqref{eq:sp}
hold true, so this region is either a triangle or a single point. 
Since there are always at least two different embeddings, namely $S_0$ and the $(N-1)$-dimensional regular simplex, this set must be a triangle whose vertices are the intersections of any two of the three lines defining the inequalities. 

Next we note that these intersection points precisely represent $S_0$, $S_1$, and $S_2$ so they are $(a_0, b_0)$, $(a_1, b_1)$, and $(a_2, b_2)$.
Indeed, project the basis orthogonally on one of the spaces ${\bf1},E_1,E_2$ and denote the (normalized)
resulting set by $X.$ The eigenvectors of this projection, corresponding to the two other spaces have zero eigenvalues. 
Subsequently,  the eigenvalues of these vectors for the Gram matrix $G=X^t X$ are also zero, 
which turns two of the inequalities in \eqref{eq:sp} into equalities.

Any other pair $(a,b)$ can be represented as $(a,b) = \alpha^2 (a_0, b_0) + \beta^2 (a_1, b_1) + \gamma^2 (a_2, b_2)$, where $\alpha^2+
\beta^2+\gamma^2=1$ and $\alpha, \beta, \gamma$ are non-negative. Now note that the set 
 $\{x_i: 1\leq i \leq N\}$ such that $x_i=\alpha x_i(0) + \beta x_i(1) + \gamma x_i(2)$, where the set of all vectors $x_i(0)$ forms $S_0$, 
 the set of all $x_i(1)$ forms $S_1$, and the set of all $x_i(2)$ forms $S_2$ in mutually orthogonal unit spheres, 
 gives a two-distance spherical embedding of $\Gamma_1$ with inner products $a$ and $b.$ Moreover, any such embedding is completely
 determined by its Gram matrix, and therefore, this gives a description of all spherical two-distance embeddings of $\Gamma_1$. This completes the
 proof.
\end{proof}

Proposition \ref{prop:srg-embed} entails the following description of two-distance $2$-designs.

\begin{theorem}\label{thm:2-design}
Any spherical two-distance $2$-design $S=\{x_1,\ldots,x_N\}$ with graph $\Gamma_1$ for one of the distances is either $S_1(\Gamma_1)$ or $S_2(\Gamma_1)$, or a regular $(N-1)$-dimensional simplex.
\end{theorem}

\begin{proof} We begin with the representation of the vectors $x_i$ given by \eqref{eq:a}.
Note that since $\sum_{i=1}^N x_i = 0$, the coefficient $\alpha$ must be 0.
If one of $\beta$ or $\gamma$ is $0$, then $S$ is either $S_1$ or $S_2$. 
The remaining case is when they are both positive. 
In this case the set $S$ is $(n_1+n_2)$-dimensional, so it must satisfy the tight-frame condition \eqref{eq:frame-bounds}-\eqref{eq:A} for any $x\in \reals^{n_1+n_2}$: 
   \begin{equation}\label{eq:12}
\frac N {n_1+n_2} ||x||^2 = \sum\limits_{i=1}^N \langle x, x_i \rangle ^2.
   \end{equation}
    Now we express $x$ as the sum of $x(1)$ and $x(2)$, where $x(1)$ belongs to the space $\reals^{n_1}$ that contains all the vectors 
    $x_i(1),$ and $x(2)$ belongs to the space $\reals^{n_2}$ containing all $x_i(2)$. 
    Since $S_1$ and $S_2$ form $2$-designs, 
    they must satisfy the tight-frame condition, namely
    $$
    \frac N {n_j} ||x(j)||^2 = \sum\limits_{i=1}^N \langle x(j), x_i(j) \rangle ^2, \quad j=1,2.
    $$ 
Using~\eqref{eq:a} and~\eqref{eq:12}, we obtain:
    \begin{align*}
    \frac N {n_1+n_2} (||x(1)||^2 + ||x(2)||^2) &= \sum\limits_{i=1}^N (\beta \langle x(1),x_i(1)\rangle + \gamma \langle x(2),x_i(2)\rangle)^2 \\
    &= \beta^2 \frac N {n_1} ||x(1)||^2 + \gamma^2 \frac N {n_2} ||x(2)||^2 +2\beta\gamma \sum\limits_{i=1}^N \langle x(1),x_i(1)\rangle \langle x(2),x_i(2)\rangle.
    \end{align*}
This equality must hold for any $x(1)$ and $x(2),$ so $\beta^2=\frac {n_1} {n_1+n_2}$ and 
$\gamma^2 = \frac {n_2} {n_1+n_2}$. To show that with these values of $\beta$ and $\gamma$ the set $S$
forms a $2$-design we just need to explain why $\sum\limits_{i=1}^N \langle x(1),x_i(1)\rangle \langle x(2),x_i(2)\rangle$ is always $0$. 
Refer to the definition of $S_1$ and $S_2$ and let their ambient spaces be ${E_1}$ and ${E_2}$. 
Then the vector with components $\langle x(1),x_i(1)\rangle$ is just $\Phi_1 x(1)$ times a normalizing coefficient, 
and the vector with components $\langle x(2),x_i(2)\rangle$ is $\Phi_1 x(2)$ with its normalizing coefficient. 
The first vector belongs to ${E_1}$ and the second vector belongs to ${E_2}$ so they must be orthogonal.

A regular $(N-1)$-dimensional simplex is obviously a $2$-design and can be considered as a two-distance embedding of 
$\Gamma_1$ with equal distances. 
Since $S_1$ and $S_2$ are not $(N-1)$-dimensional, the 
third $2$-design that we constructed must be a regular simplex (recall that $n_1+n_2=N-1$). 
This observation finishes the proof of the theorem.
\end{proof}

\begin{remark}
The regular simplex can be constructed similarly to $S_1$ and $S_2$: it is obtained by finding orthogonal projections of (the basis vectors of) $\reals^N$ on  $E_1\cup E_2$ and normalizing to get unit lengths.
Another simplex is given by the orthonormal basis itself which represents a trivial projection.
\end{remark}

\begin{proof} {\em Proof of Theorem \ref{thm:main}.} We now recap the arguments that lead to the classification of all non-equiangular two-distance tight frames in Theorem \ref{thm:main}. Let $S$  be such a frame and assume that  $a$ and $ b$ are the two distinct inner products of the vectors in $S$. Then, $a^2-b^2\neq 0$. First, by Proposition~\ref{prop:tf-design}, $S$ is  either a $n$ dimensional spherical $2-$design or similar to $(n-1)$ dimensional spherical $2-$design. On account of Prop.~\ref{prop:2-tight-srg} the graphs defined by the Gram matrix of a non-equiangular two-distance FUNTF are strongly regular, so we need to describe all spherical two-distance embeddings of SRGs and check if they satisfy the design condition. We show in Theorem~\ref{thm:2-design}  that all such embeddings are of the Delsarte-Goethals-Seidel type, and yield spherical two-distance 2-designs. Since each such design gives rise to two FUNTFs, this completes the classification.

%

\end{proof}

\vspace*{.1in} The results established above enable us to construct large classes of two-distance tight frames. For brevity we
write $\text{FUNTF}(n,N,N_a,a,b)$ to refer to a two-distance tight frame in $n$ dimensions, with $N$ points, inner products $b<a,$ and
with $N_a$ entries $a$ in each row of $G$.
We give a few examples of 2-distance frames derived from the table of strongly regular graphs in \cite[pp.143ff]{Brouwer12}. Many more
examples can be easily obtained using the described recipe.

\vspace*{.2in}\begin{center}{\small\begin{tabular}{l||l@{\hspace*{.15in}}|@{\hspace*{.15in}}ll}
$\text{SRG}(N,k,\lambda,\mu)$&  2-design FUNTF$(n,N,N_a,a,b)$ & shifted 2-design FUNTF$(n,N,N_a,a,b)$\\
\hline
 $(10,6,3,4)$ &$(4,10,6,\nicefrac16,-\nicefrac23),\; (5,10,3,\nicefrac13,-\nicefrac13)$ 
               &  $(5,10,6,\nicefrac13,-\nicefrac13),\; (6,10,3,\nicefrac49,-\nicefrac19)$ \\[.03in] 
$(15,8,4,4)$ &$(5,15,8,\nicefrac14,-\nicefrac12),\; (9,15,8,\nicefrac16,-\nicefrac14)$ 
               & $(6,15,8,\nicefrac38,-\nicefrac14),\; (10,15,6,\nicefrac14,-\nicefrac18)$ \\ [.03in] 
$(16,10,6,6)$ &$(5,16,10,\nicefrac15,-\nicefrac35),\; (10,16,5,\nicefrac15,-\nicefrac15)$ 
               &  $(6,16,10,\nicefrac13,-\nicefrac13),\;(11,16,5,\nicefrac3{11},-\nicefrac1{11})$\\
               
\end{tabular}
}
\end{center}

\vspace*{.2in}


\begin{thebibliography}{99}

\bibitem{ban05}
E. Bannai and Et. Bannai, \emph{A note on the spherical embeddings of strongly regular graphs},
European J. Comb., \textbf{26} (2005), 1177--1179.

\bibitem{ban14}
E. Bannai, T. Okuda, and M. Tagami, \emph{Spherical designs of harmonic index $t$},
J. Approximation Theory, in press; Preprint available arXiv:1306:5101.

\bibitem{barg13}
A. Barg and W.-H. Yu, \emph{New bounds for spherical two-distance
sets}, Experimental Mathematics, \textbf{22}, no. 2, (2013),
187--194.
\bibitem{barg14}
A. Barg and W.-H. Yu, \emph{New bounds on equiangular lines}, in: {Discrete Geometry and Algebraic Combinatorics,} 
A. Barg abd O. Musin, eds., (Contemporary Mathematics, vol.~{625}), Amer. Math. Soc., Providence, RI, 2014, pp.~111--121.

\bibitem{ben03}
J. J. Benedetto and M. Fickus, \emph{Finite normalized tight
frames}, Advances in Computational Math, \textbf{Vol 18},
nos.~2-4, (2003), 357--385.

\bibitem{bod07}
B. Bodmann, P. Casazza, and R. Balan, \emph{Frames for linear
reconstruction without phase}, Proc. 42nd IEEE Annual Conference
on Information Sciences and Systems (CISS 2008), Princeton, NJ,
March 19-21, 2008, pp. 721--726.


\bibitem{bro12}
A. E. Brouwer and W. H. Haemers, \emph{Spectra of graphs}, Springer, New York e.a., (2012).

\bibitem{ck12}
P.~G.~Casazza and G.~Kutyniok (Editors), ``Finite Frame Theory,'' 
Birkh\"auser, Boston (2012).

\bibitem{Brouwer12} A. E. Brouwer and W. H. Haemers, {Spectra of graphs,} Springer, New York e.a., 2012.





\bibitem{del77b}
P.~Delsarte, J.~M. Goethals, and J.~J. Seidel, \emph{Spherical
codes and designs}, Geometriae Dedicata \textbf{6} (1977),
363--388.

\bibitem{Godsil2001} C. Godsil and G. Royle,  Algebraic Graph Theory, Springer, New York 2001.

\bibitem{hklw}
D. Han, K. Kornelson, D. Larson,  and E. Weber, \textit{Frames for
Undergraduates}, American Mathematical Society, Providence, RI,
2007.


\bibitem{hol04}
R. B. Holmes and V.I. Paulsen, \emph{Optimal frames for erasures},
Linear Alg. and Application, \textbf{377} (2004), 31-51.


\bibitem{koche1}
J.~Kova\v{c}evi\'{c} and A.~Chebira, {\it Life Beyond Bases: The
Advent of Frames (Part I-II),} Signal Processing Magazine, IEEE,
Volume {\bf 24}  (2007), 86--104 and 115--125.


\bibitem{lar77}
D.~G. Larman, C.~A. Rogers, and J.~J. Seidel, \emph{On
two-distance sets in
  {E}uclidean space}, Bull. London Math. Soc. \textbf{9} (1977), 261--267.

\bibitem{lem73}
P. W. H. Lemmens and J.~J. Seidel, \emph{Equiangular lines},
Journal of Algebra \textbf{24} (1973), 494--512.

\bibitem{lint66}
J. H. van Lint and J.J. Seidel, \emph{Equiangular point sets in
elliptic geometry} Proc. Nedert. Akad. Wetensh. Series \textbf{69}
(1966), 335-348.



\bibitem{mus09a}
O. R. Musin, \emph{Spherical two-distance sets}, J. Combin. Theory
Ser. A \textbf{116}, no.~4 (2009), 988--995.



\bibitem{neu81}
A. Neumaier, \emph{Distance matrices, dimension, and conference
graphs}, Indag. Math., \textbf{43}, no.~4 (1981), 385--391.

\bibitem{noz12} H. Nozaki and M. Shinohara, \emph{A geometrical characterization
of strongly regular graphs}, Linear Alg. Appl. \textbf{437} (2012), 2587--2600.

\bibitem{Seidel73}
J.~J. Seidel, \emph{A survey of two-graphs}, Colloquio
Internazionale sulle Teorie Combinatorie (Rome, 1973), Tomo I, pp.
481--511. Atti dei Convegni Lincei, No. 17, Accad. Naz. Lincei,
Rome, 1976.

\bibitem{sus07}
M. A. Sustik, J. A. Tropp, I. S. Dhillon and R.W. Heath, Jr.,
\emph{On the existence of equiangular tight frames}, Linear Alg.
and Applications \textbf{426}, no.~2-3 (2007),  619-635.


\bibitem{stro03}
T. Strohmer and R. W. Heath, \emph{Grassmannian frames with
applications to coding and communication}, Appl. Comp. Harmonic
Anal. \textbf{14}, no.~3 (2003), 619-635.

\bibitem{wal09}
S. Waldron, \emph{On the construction of equiangular frames from
graph}, Linear Alg. and its Applications \textbf{431}, no. 11
(2009), 2228-2242.


\end{thebibliography}
\end{document}